\documentclass[12pt]{article}
\usepackage{amsmath,amssymb,amsthm,fullpage}
\usepackage{mathtools}               
\usepackage{natbib}        
\usepackage[normalem]{ulem}          
\RequirePackage[colorlinks,citecolor=blue,urlcolor=blue]{hyperref} 
\usepackage{lmodern}

\newtheorem{lem}{Lemma}
\newtheorem{thm}{Theorem}
\newtheorem{cor}{Corollary}

\title{Nonexistence of Consecutive Powerful Triplets Around Cubes with Prime-Square Factors}
\author{Jialai She}
\date{}

\begin{document}

\maketitle

\begin{abstract}
The Erd\H{o}s-Mollin-Walsh  conjecture, asserting the nonexistence of three consecutive powerful integers, remains a celebrated open problem in number theory.  A natural line of inquiry, following recent work by Chan (2025), is to investigate potential counterexamples centered around perfect cubes, which are themselves powerful.  This paper establishes a new non-existence result for a family of such integer triplets with distinct structural constraints, combining techniques from modular arithmetic, $p$-adic valuation, Thue equations, and the theory of elliptic curves.

\end{abstract}

\section{Introduction}
A positive integer is called a \emph{powerful number} if each of its prime factors appears with an exponent of at least two. Every powerful number $n$ admits a  \emph{unique} representation  as
\begin{align}
n=a^2b^3,
\end{align} where $a, b\in \mathbb Z$  and $b$ is  square-free   (meaning  that it is not divisible by any perfect square other than 1) (see \cite{golomb1970powerful} for example).
The    Erd\H{o}s-Mollin-Walsh  conjecture \citep{erdHos1976problems,mollin1986powerful} asserts that no three consecutive integers are all powerful.

 A natural and  interesting case arises when  the triplet is centered at a perfect cube $x^3$, itself always powerful. That is, does the triplet $(x^3-1, x^3, x^3+1)$ ever consist entirely of powerful numbers?  Recently, \cite{Chan2025} resolved a  subcase by proving no such triplets exist when
\begin{align}
x^3 - 1 = p^3 y^2 , \ \  x^3 + 1 = q^3 z^2, \label{Chaneq}
\end{align}
for primes $p, q$ and   integers $x,y,z>0$.


In this paper, we prove the non-existence of a new family of triplets in this setting, extending Chan's result while  addressing distinct constraints.  Our main result is as follows.

\begin{thm}\label{thm:main}
  There exist no consecutive powerful numbers of the form  \[
    x^3 - 1 = p^2\,a^3, \qquad x^3,
    \qquad
    x^3 + 1 = q^2\,b^3.
  \]
  where   \(p,q\) are primes and  \(a,b,x\) are integers.
\end{thm}
Notice that the powers in Theorem 1 differ from those  in \eqref{Chaneq}.
Furthermore, unlike \cite{Chan2025}, we do not need to impose the restriction that  variables $x, a,   b$ are positive.
This result establishes the non-existence of a class of consecutive powerful triplets not covered by previous literature.

\begin{cor}\label{mycor}
For any primes $p, q$ and any integers
  $x, a$ with $a \ne 0$, the equation $ x ^6 - 1 =p^ 2 q^2 a^3$ has no solution. 
\end{cor}

\section{Proof of the Main Result}
First, let us introduce some lemmas.
\begin{lem}\label{split} Let \(p\) be a prime and \(R,S,C \) be integers that satisfy
\[
R\,S  =  p^2\,C^3,
\]
and set \(g=\gcd(R,S)\).
If \(g=1\), then  one of \(R,S\) is a perfect cube and the other is \(p^2\) times a perfect cube. If  \(g\) is a prime, then there exist $C_1, C_2\in \mathbb Z$ such  that either  $(R, S)=(g C_1^3, g C_2^3)$, $\{R, S\}=\{g C_1^3, g^2 p^2 C_2^3\}$,  or $\{R, S\}=\{g p^2 C_1^3, g^2   C_2^3\}$.
\end{lem}
\begin{proof}
For $g=1$, if   $p \nmid C$,
any prime factor $r\neq p$ of $C$ divides exactly one of $R,S$, and if $p \mid C$, then all $p$ factors in $p^2 C^3$ must be contained entirely within either $R$ or $C$, but not both.  Assume  $g$ is a prime.  If $g=p$, then both $R/p$ and $S/p$ are perfect cubes. Otherwise, the two  factors   can be written as  $g C_1^3$ and $g^2 p^2 C_2^3$, or $g p^2 C_1^3$ and $g^2   C_2^3$, up to ordering. The conclusion follows.
\end{proof}

\begin{lem}\label{lem:modell}
The Diophantine equation
$ 
u^2 + u + 1 = 3v^3
$ 
 has exactly two integer solutions  $(u,v) \in \{ (-2,1), (1,1) \}$. The Diophantine equation
$ 
u^2 - u + 1 = 3 v^3 \label{modell-ex1}
$ 
 has exactly two integer solutions  $(u,v) \in \{ (2,1), (-1,1) \}$. \end{lem}

\begin{proof}
The core of the proof is to convert the original equations into the standard form of a Mordell curve through proper transformations.
Taking  the second equation 
as an example. We first multiply the equation by $3^2$ and substitute $u' = 3u$ and $v' = 3v$, which yields
$
(u')^2 - 3u' + 9 = (v')^3.
$
Multiplying by 4 and completing the square on the left side allows for the substitution $u'' = 2u' - 3$ and $v'' = v'$, leading to
$
(u'')^2 + 27 = 4 (v'')^3.
$
Finally, multiplying both sides by $2^4$ and letting $x = 4v''$ and $y = 4u''$ reduces the equation to the canonical Mordell form:
$
y^2 = x^3 - 432.
$
The integer solutions to this well-known curve are $(x,y) = (12, \pm 36)$ (see \cite{gebel1998mordell} for example). Working backwards   gives the two solutions for $(u, v)$. The other equation can be handled by substituting    $u \mapsto - u$.
\end{proof}
\begin{lem}\label{cubic-mordell}
The integer solutions for $u$ to the equation
$ 
u^2 + u + 1 = v^3
$ 
are given by \(u = -19, -1, 0, 18\).
Similarly, the integer solutions for $u$ to
$ 
u^2 - u + 1 = v^3
$ 
are \( -18,  0,  1,  19\).\end{lem}
\begin{proof}
The first equation can be handled using the transformations and Mordell curve techniques in Lemma \ref{lem:modell}, or directly by citing the corollary of \cite{TZANAKIS1984}. The second equation then follows by the substitution $u \mapsto - u$.
\end{proof}


\begin{lem}\label{lem:gcd}
For any integer \(x\),
$ 
  \gcd\bigl(x-1, x^2+x+1\bigr)
   = \gcd(x-1,3)$ and $
  \gcd\bigl(x+1, x^2-x+1\bigr)
   = \gcd(x+1,3).
$ 
\end{lem}

\begin{proof}
  Writing
  $
    x^2+x+1
    = (x-1)(x+2) + 3$,
 $   x^2 - x + 1
    = (x+1)(x-2) + 3.
  $
   The Euclidean  algorithm on polynomials yields the conclusion.  \end{proof}

\begin{lem}\label{lem:cube-diff}
  The Diophantine equation
$
    u^3 - v^3 = 2
$ 
  has the unique integer pair solution $(1,-1),$ and the Diophantine equation
$ 
    u^3 - v^3 = 1
$ 
  has the  integer solutions $(1,0)$ and $(0,-1)$.
\end{lem}

\begin{proof}
    We may assume $u,v$ are both nonzero.
    Consider $u^3 - v^3 = 2$. Since  $u > v$, $u - v> 0, u^2 + uv + v^2 > 0, $  either $u - v = 2, u^2 + uv + v^2 = 1,$ or $u - v = 1, u^2 + uv + v^2 = 2.$
    Simple calculation yields the unique solution $(1, -1)$.
The other equation can be treated similarly.
\end{proof}



We are now prepared to establish the main result.
Suppose for contradiction that
\begin{align}
  x^3 - 1 = p^2\,a^3,
  \quad
  x^3 + 1 = q^2\,b^3, \label{contrass}
\end{align}
with integers \(x,a,b \)  and  primes \(p,q\).
Set
$ 
g_{-}=\gcd(x-1,x^{2}+x+1), 
$ $
g_{+}=\gcd(x+1,x^{2}-x+1).
$ 
By Lemma \ref{lem:gcd},
\[
g_{-}=
\begin{cases}
3 & x\equiv1\pmod3,\\
1 & \text{otherwise,}
\end{cases}\qquad
g_{+}=
\begin{cases}
3 & x\equiv2\pmod3,\\
1 & \text{otherwise.}
\end{cases}
\]
We proceed by casework since  the pair $(g_{-},g_{+})$ can be $(1,1)$, $(1,3)$, or $(3,1)$ only.
\\

\noindent\textbf{Case 1: $(g_{-},g_{+}) = (1,1)$, $x \equiv 0 \pmod{3}$.}
By Lemma \ref{split}, we have the following possibilities
\begin{align*}
\begin{split}
  &\text{(i)}\quad x-1 = u^3, x^2+x+1 = p^2 v^3,\\
  &\text{(ii)}\quad x-1 = p^2 u^3, x^2+x+1 = v^3,
\end{split} \quad \text{ and}
\begin{split}
  & \text{(a)}\quad x+1 = s^3,  x^2-x+1 = q^2 t^3,\\
  &\text{(b)}\quad x+1 = q^2 s^3, x^2-x+1 = t^3,
\end{split}
  \end{align*}
where $u,v,s,t$ are  integers. We explore each subcase below.

\begin{itemize}
\item \((i)+(a).\)
We obtain $x-1 = u^3,  x+1 = s^3
  $, or
  $s^3 - u^3 = 2.$
By Lemma \ref{lem:cube-diff}, $s = 1, u = -1.$ But then $x = 0$ and   there exists no prime $p$ satisfying \eqref{contrass}.
\item $(i)+(b).$
Applying Lemma~\ref{cubic-mordell} to   \(x^2 - x + 1 = t^3\) yields \(x = -18,  0,  1,  19\). However, none of these values satisfies  \(x + 1 = q^2 s^3\) under the given constraints (for example, taking $x=19 $ leads to $ q^2 s^3 = 20$ and thus $ q = 2$, but no integer $s$ exists).

\item  $(ii)+(a).$
This subcase can be treated similarly by applying Lemma~\ref{cubic-mordell} to the equation \(x^2 + x + 1 = v^3\).

\item  $(ii)+(b).$
In this subcase, $x^2 - x + 1$ and $ x^2 + x + 1$ are perfect cubes. Lemma  \ref{cubic-mordell}  forces $x=0$, but
 then no prime $q$ exists satisfying \eqref{contrass}.

\end{itemize}

\noindent\textbf{Case 2: $(g_{-},g_{+}) = (1,3)$, $x\equiv2 \pmod{3}$.}
Let $v_p (n)$ denote the $p$-adic valuation of $n.$

First, applying the Lifting-the-Exponent lemma (see, e.g., Theorem 1.37 of \cite{Kaya2023}) to $v_3(x^3+1)$ gives $v_3 (x + 1) + 1 = v_3 (x^3 + 1) = v_3 (x + 1) + v_3(x^2 - x + 1),$ or
\begin{align}
v_3 (x^2 - x + 1) = 1. \label{factor3}
\end{align}

We split two subcases based on $q$.
\begin{itemize}
    \item $q = 3:$ For  $r \neq 3$   as  an arbitrary prime factor of $x^2 -  x + 1,$  since $g_+ = 3$, we have $v_r (x^2 - x + 1) = v_r ((x + 1)(x^2 - x + 1)) = v_r (9b^3) = 3 v_r (b)$.   Combined with \eqref{factor3},   $x^2 - x + 1$ can be written as $3 \prod_{i = 1} ^{n} r_i ^{3\alpha_i},$ or $3s^3$ for some $s\in \mathbb Z$. Applying Lemma \ref{lem:modell} gives $x = -1$ or $2.$ Yet neither yields a valid solution for \eqref{contrass} with prime $p.$

    \item $q \neq 3:$ Here, we have   $v_3 (x^3 + 1) = v_3 (q^2 b^3)=3  v_3 (b)$. Combined with    \eqref{factor3},     $v_3 (x^3 + 1) \ge  3$, from which it follows that  $v_3 (x+1) = v_3 ((x^3 - 1)/(x^2 - x + 1)) \ge  3 - 1 = 2$, and thus
%
%
%
\begin{align}x \equiv -1 \pmod{9}. \label{xmod}
\end{align}

    Next, applying Lemma $\ref{split}$ to $(x-1)(x^2 + x + 1) = p^2 a^3$, we have   two possibilities: (i) $x - 1 = u^3, \, x^2 + x + 1 = p^2 v^3,$ (ii) $x - 1 = p^2 u^3, \, x^2 + x + 1 = v^3.$  Combining   (i) and   \eqref{xmod} yields $ u^3 \equiv - 2 \pmod 9$, which  is impossible. For (ii),     \eqref{xmod}    and Lemma $\ref{cubic-mordell}$  force $x = -19$ or $-1.$ Neither yields a valid solution for \eqref{contrass}  with prime $p.$

\end{itemize}

\noindent\textbf{Case 3: $(g_{-},g_{+}) = (3,1)$, $x\equiv1 \pmod{3}$.}
This case is analogous to the previous one, with $x$ replaced by $-x$. The result  follows from a combination of $p$-adic analysis and Lemmas \ref{split}, \ref{lem:modell}, and \ref{cubic-mordell}.
\\

\section{Proof of the Corollary}
The proof of Corollary \ref{mycor} combines the arguments for our main theorem with that  of Corollary 1 of \cite{Chan2025}. We begin with a few preparatory lemmas.

\begin{lem}\label{lem:corrv3}
If \(3\mid x-1\), then \(v_3(x^2+x+1)=1\); if \(3\mid x+1\), then \(v_3(x^2-x+1)=1\).
\end{lem}
\begin{proof}
The argument is analogous to the one used in the main theorem's proof.
\end{proof}

\begin{lem}\label{lem:cube-2diff}
      The only   integer solutions of the Diophantine equation
$
    u^3 - 2v^3 = 1
$ 
are $(1,0)$ and $(-1,-1)$.
\end{lem}

\begin{proof}
Both \((1,0)\) and \((-1,-1)\) satisfy \(u^3-2v^3=1\). By the Delone-Nagell theorem (see, e.g., Theorem V, \S72 of \cite{DeloneFaddeev1964}), there is at most one solution in addition to \((1,0)\). The conclusion follows.
\end{proof}

Although the following result is likely known, we were unable to locate a specific reference for these parameters. For the sake of completeness, we provide a brief, self-contained proof based on Lemma~\ref{lem:modell}.

\begin{lem}\label{lem:cube-4diff}
For \(d\in\{4,18,36\}\), the Diophantine equation \(u^{3}-d v^{3}=1\) has the unique integer solution \((1,0)\).

\end{lem}

\begin{proof}
First, for \(d=4\), we consider the equation \(u^3 - 4v^3 = 1\). Reducing this modulo 9 implies that   \(u^3 \equiv 1 \pmod 9\) and \(v^3 \equiv 0 \pmod 9\), and so  \(u \equiv 1 \pmod 3\) and  \(v \equiv 0 \pmod 3\).
Let \(v=3v_0\) for some integer \(v_0\). The equation becomes
$  (u-1)(u^2+u+1) = 4 \cdot 3^3 \cdot v_0^3$.
  Lemma~\ref{lem:corrv3} gives \(v_3(u^2+u+1)=1\), from which it follows that  \(\gcd(u-1, u^2+u+1)=3\). We can therefore set \(u-1=3a\) and \(u^2+u+1=3b\) with $\gcd(a, b) =1$, and obtain  \(ab = 12v_0^3\).
Since \(u^2+u+1\) is odd,     \(b\) must also be odd. As \(v_3(3b)=1\), $3\nmid b$.  It follows that  \(b=s^3\), and   \(u^2+u+1 = 3s^3\). By Lemma~\ref{lem:modell},   \(u=1,-2\). Noting \(u\) must be odd, the conclusion follows.

 Similarly, for $d=18$, the equation $u^3-18v^3=1$ taken modulo 9 gives $u^3\equiv 1 \pmod 9$, which implies $u\equiv 1 \pmod 3$.
By Lemma~\ref{lem:corrv3}, \(v_3(u^2+u+1)=1\), and thus    \(\gcd(u-1, u^2+u+1)=3\). Setting $u-1=3a$ and $u^2+u+1=3b$ with $\gcd(a,b)=1$ results in $ab=2v^3$. Since $b$ must be odd and is coprime to $a$,   $b=s^3$, which yields $u^2+u+1 = 3s^3$. The unique integer solution $(1,0)$ again follows by Lemma~\ref{lem:modell}. The case for $d=36 $ follows from an identical argument, applying a modulo 9 reduction along with Lemma~\ref{lem:corrv3} and Lemma~\ref{lem:modell}.    \end{proof}

%
%
%

We now prove Corollary \ref{mycor} by contradiction. 
The equation in the corollary can be written as
\begin{align} \big( x^3 - 1 \big) \big( x^3 + 1 \big) = p^2 q^2 a^3. \label{x6eqinpf}
\end{align}
 Since   $\gcd(x^3 - 1, x^3 + 1) \mid   x^3 + 1 - (x^3 - 1)=2$, $\gcd(x^3 - 1, x^3 + 1)=1$ or $2$.

For $\gcd(x^3 - 1, x^3 + 1) = 1$,  there are  two possibilities for the  factors on the left-hand side of \eqref{x6eqinpf}. Assume  $p^2 q^2$ divides one of the factors.
This implies that the other factor must be a perfect cube.
If   $ x^3 + 1 = a_1^3$ or      $x^3 - 1 = a_1^3$, $x$ has no valid solution   by Lemma \ref{lem:cube-diff}.
Therefore,
 $p^2$ divides one factor and $q^2$ divides the other;
without loss of generality, assume
$
     x^3 - 1 = p^2 a_1^3 ,
    x^3 + 1 = q^2 a_2^3
$.
But this  contradicts our main theorem.

Next, consider    $\gcd(x^3 - 1, x^3 + 1) = 2$. If one of the primes is 2, say   $p=2$, then we have two possible systems of equations:  \begin{align}
\begin{cases}
x^3-1 = 2 q^2 u^3\\
x^3+1 = 2 v^3,
\end{cases}
\text { or }
\begin{cases}
x^3-1 = 2   u^3\\
x^3+1 = 2 q^2 v^3.
\end{cases}
\end{align}
 By Lemma \ref{lem:cube-2diff}, $x = \pm 1$, which violates $a \ne 0$.

For the remainder of the proof,  assume     $p \ne 2, q \ne 2$, and    \eqref{x6eqinpf} becomes  $(x^3-1)(x^3+1) = 8 p^2 q^2 b^3$. One possibility is  that    $p^2 q^2$ divides one of the factors on the left.  This leads to four  subcases, each of which yields a contradiction by applying either Lemma \ref{lem:cube-2diff} or Lemma \ref{lem:cube-4diff}:
\begin{itemize}
\item $x^3 - 1 = 2 p^2 q^2 u^3,
x^3 + 1 = 4 v^3$. By Lemma \ref{lem:cube-4diff},   $x=-1$, violating $a\ne 0$.

\item $x^3 - 1 = 4 p^2 q^2 u^3,
x^3 + 1 = 2   v^3
$. By Lemma \ref{lem:cube-2diff}, $x = \pm 1$,  violating $a\ne 0$.
\item $x^3 - 1 = 2 u^3,
x^3 + 1 = 4   p^2 q^2 v^3
$. By Lemma \ref{lem:cube-2diff}, $x = \pm 1$,  violating $a\ne 0$.
\item $x^3 - 1 = 4 u^3,
x^3 + 1 = 2   p^2 q^2 v^3$. By Lemma \ref{lem:cube-4diff},   $x=1$, violating $a\ne 0$.
\end{itemize}
Therefore,
 $p^2$  must divide one of $x^3-1$ and $x^3+1$, and $q^2$ must divide the other.

It remains to  study the system \begin{align}
x^3 - 1 = 2 p^2 u^3, \quad
x^3 + 1 = 4   q^2 v^3. \label{subcaseeqcor}
\end{align}
  Indeed, when  the factors of 2 and 4 are swapped,    \(x^{3}-1=4p^{2}u^{3}\) and \(x^{3}+1=2q^{2}v^{3}\) reduce to the form in \eqref{subcaseeqcor} under the change of variables   \(x\mapsto -x,  u\mapsto -v,  v\mapsto -u,  p\mapsto q,   q\mapsto  p\). The equations in  \eqref{subcaseeqcor} resemble those in \eqref{contrass}, but are distinct:  for example, $x^3-1$ here is not necessarily a powerful number. Nevertheless, the same proof strategy used in the previous section can be applied.

Recall the definitions of $g_-, g_+$ in the proof of the main theorem. We proceed in a similar manner.

\noindent\textbf{Case 1: $(g_{-},g_{+}) = (1,1)$.}
The factorization of the first equation  in \eqref{subcaseeqcor}  yields four possibilities:    (i) $ x-1 = 2t_1^3 , x^2+x+1 = p^2 t_2^3$, (ii) $ x-1 =2 p^2 t_1^3, x^2+x+1 = t_2^3 $, (iii) $ x-1 =  t_1^3 , x^2+x+1 = 2p^2 t_2^3$, or (iv) $ x-1 =  p^2 t_1^3, x^2+x+1 =2 t_2^3 $, and similarly, the second equation gives:    (a) $ x+1 = 4 t_3^3, x^2-x+1 = q^2 t_4^3$, (b) $  x+1 = 4 q^2 t_3^3, x^2-x+1 = t_4^3$, (c) $ x+1 =   t_3^3, x^2-x+1 = 4q^2 t_4^3$, (d) $  x+1 =  q^2 t_3^3, x^2-x+1 = 4 t_4^3$,
where $ t_i$ are  integers.

Since \(x^{2}\pm x+1=x(x\pm1)+1\) is always odd,  (iii), (iv), (c),   and  (d) are impossible. Among the remaining combinations, those involving \((\mathrm{ii})\) or \((\mathrm{b})\) can be  excluded by Lemma~\ref{cubic-mordell}. For example, under \((\mathrm{i})+(\mathrm{b}),\) we have \(x\) odd, and Lemma~\ref{cubic-mordell} forces \(x\in\{1,19\}\); \(x=1\) gives \(a=0\) (a contradiction), while \(x=19\) yields \(x+1=20\), so no prime \(q\) exists satisfying \eqref{subcaseeqcor}.

 Thus only \((\mathrm{i})+(\mathrm{a})\) remains. In this case,
$
2t_3^{3}-t_1^{3}=1.
$
By Lemma~\ref{lem:cube-2diff}, \(t_1=\pm1\), hence \(x\in\{3,-1\}\), but neither value produces a prime \(p\) satisfying \eqref{subcaseeqcor}.

\noindent\textbf{Case 2: $(g_{-},g_{+}) = (3,1)$.}
  If $p=3$,   Lemma \ref{lem:cube-4diff} implies $x=1$, violating  $a\ne 0$. Assume $p\ne 3$. The case condition requires $x\equiv 1 \pmod 3$,   and applying Lemma \ref{lem:corrv3} gives $v_3(x^2+x+1)=1$. Combining this with the parity argument from Case 1, it suffices to   analyze the   factorizations: (i) $x-1 = 18 t_1^3, x^2+x+1  = 3 p^2  t_2^3$,  or (ii)    $x-1 = 18 p^2 t_1^3,    x^2+x+1  = 3 t_2^3$, and (a) $x+1 =  4  t_3^3, x^2-x+1  = q^2  t_4^3$,  or (b) $x+1 =  4 q^2 t_3^3, x^2-x+1  =   t_4^3$.
By a similar line of reasoning, applying Lemma \ref{lem:modell} and Lemma \ref{cubic-mordell} eliminates all but one nontrivial combination,   (i) + (a),    from which it follows   that  $2 t_3^3 - 9 t_1^3 = 1 $. Reducing the equation modulo 9 gives  $2 t_3^3 \equiv  1\pmod 9$,    a contradiction.

\noindent\textbf{Case 3: $(g_-, g_+) = (1,3)$.} If $q=3$,   Lemma \ref{lem:cube-4diff}  implies $x=-1$, violating  $a\ne 0$. Assume $q\ne 3$. Applying Lemma \ref{lem:corrv3} gives $v_3(x^2-x+1)=1$. Combining   the 3-adic valuation and the parity argument as in Case 2, the possible factorizations are:   (i) $x-1 = 2 t_1^3, x^2+x+1  =   p^2  t_2^3$,  or (ii)    $x-1 = 2 p^2 t_1^3,    x^2+x+1  =   t_2^3$, and (a) $x+1 =  36   t_3^3, x^2-x+1  = 3q^2  t_4^3$,  or (b) $x+1 =  36 q^2 t_3^3, x^2-x+1  =   3t_4^3$.
Applying Lemma \ref{lem:modell} and Lemma \ref{cubic-mordell} leaves only one nontrivial case,
   $x-1 =2 t_1^3,  x^2+x+1 = p^2 t_2^3, x+1 = 36 t_3^3,x^2-x+1 =  3 q t_4^3$, implying  $18 t_3^3 -  t_1^3 = 1 $. Then Lemma \ref{lem:cube-4diff}   gives $t_1=-1$, and thus $x=-1$, violating $a\ne 0$.

The proof is complete.

\section{Conclusion}
We have proved that no three consecutive integers centered at a perfect cube can all be powerful under the structural constraints studied here. This result extends recent advances on the Erd\H{o}s-Mollin-Walsh  conjecture by eliminating a notable     family of potential counterexamples through modular and elliptic curve methods.

One natural extension is to examine whether similar non-existence holds when centered around higher powers or other special integers, and what further constraints might eliminate all consecutive powerful numbers entirely. A related and more fundamental question that arose during our research is the following: we conjecture that for every integer $x>1$ and every integer $n>2$, the number $x^n-1$ is not powerful. This is a stronger claim than   Mih{\u{a}}ilescu's Theorem (formerly Catalan's Conjecture) \citep{mihailescu2004primary}. A proof of this general assertion would have significant implications; the specific case for $n=3$ would  resolve the question addressed in Theorem \ref{thm:main} and provide deeper insight into the structure of powerful numbers.

\section*{Acknowledgment}
The author is grateful to  Tudor Popescu for bringing Chan's work to his attention and for proposing the key conjecture addressed  in this paper.
 In an earlier draft, the author established the result of the corollary with $2x$ in place of  $x$; the author is grateful to Dr. Tsz Ho Chan for his insightful suggestion to investigate removing the factor of $2$, which was accomplished in this revision.  The author also thanks the anonymous referee for independently suggesting this same valuable improvement, providing a helpful proof sketch, and offering many stylistic suggestions that improved the exposition of the paper. The conjecture stated in the final section was first raised by the author in the personal communication with Dr. Chan.
 
\bibliographystyle{apalike}
\bibliography{nt}

\end{document}